\documentclass[11pt,a4paper]{article}

\usepackage[T1]{fontenc}
\usepackage[utf8]{inputenc}
\usepackage{lmodern}
\usepackage[a4paper,margin=27mm]{geometry}
\usepackage{amsmath,amssymb,amsthm,mathtools}
\usepackage{microtype}
\usepackage{hyperref}
\hypersetup{hidelinks}
\newtheorem{theorem}{Theorem}[section]
\newtheorem{lemma}[theorem]{Lemma}
\newtheorem{proposition}[theorem]{Proposition}

\theoremstyle{remark}
\newtheorem{remark}[theorem]{Remark}
\newcommand{\F}{\mathbb F}
\newcommand{\PP}{\mathbb P}
\newcommand{\G}{\Gamma}
\newcommand{\Gs}{\Gamma^{s}}
\newcommand{\diam}{\operatorname{diam}}
\newcommand{\dist}{\operatorname{dist}}

\title{A note on the diameter of graphs of two-dimensional simplex codes}
\author{Artur Siemaszko\\[4pt]
\normalsize Institute of Mathematics, University of Warmia and Mazury in Olsztyn\\
\normalsize S\l oneczna 54, 10-710 Olsztyn, Poland\\
\normalsize\texttt{artur@uwm.edu.pl}}
\date{}

\begin{document}
\maketitle
\vspace{-1.4em}
\begin{abstract}
Let $\Gs(2,q)$ denote the graph induced in the Grassmann graph by the $q$-ary simplex codes of dimension $2$. We determine its diameter for every prime power $q$. The cases $q=2,3$ are elementary, while the diameter for $q=4$ follows from a known description of the distance relation. We prove that the diameter equals $3$ for every $q\ge5$, so that it is $0$ for $q=2$, $2$ for $q=3$, and $3$ for all $q\ge4$. The proof uses the realization of two-dimensional simplex codes as projective lines and a ratio criterion for collinearity of simplex points. A consequence of a theorem of Marshall Hall on finite abelian groups yields a bridge lemma and gives the uniform upper bound $3$. To obtain the matching lower bound for $q\ge5$, we construct an explicit diagonal pair of simplex lines and rule out a common neighbor by two independent arguments, one based on moments and the other on products. For $q\ge7$, we also give an independent counting proof. The case $q=4$ is treated separately. In particular, for $q\ge4$ the simplex-code graph is not isometrically embedded in the ambient Grassmann graph.
\end{abstract}

\bigskip

\section{Introduction}
Let $V$ be an $n$-dimensional vector space over $\F_q$. The Grassmann graph $\G_k(V)$ has the $k$-dimensional subspaces of $V$ as vertices; two distinct vertices are adjacent when their intersection has dimension $k-1$. Equivalently, its vertices may be viewed as $q$-ary linear $[n,k]$ codes, with adjacency corresponding to the largest possible intersection of two distinct codes. Induced subgraphs obtained by restricting to non-degenerate, projective, simplex, or other distinguished classes of codes have been studied in a series of papers; see, for example, \cite{KP2016,KPP2018,Pankov2023}.

A $q$-ary simplex code of dimension $k$ has length $[k]_q=(q^k-1)/(q-1)$, and every non-zero codeword has Hamming weight $q^{k-1}$.

Following Kwiatkowski and Pankov \cite{KP2024}, we use $\Gs(k,q)$ for the subgraph induced by the simplex codes in the corresponding Grassmann graph. They proved that this graph is connected and noted that the distance problem remains open in general.

We are concerned with $k=2$, where the ambient space is $V=\F_q^{q+1}$. In projective language, a \emph{simplex point} is a point of $\PP(V)$ represented by a vector with exactly one zero coordinate, and a \emph{simplex line} is a projective line corresponding to a two-dimensional simplex code. Thus $\Gs(2,q)$ is the intersection graph of simplex lines: two vertices are adjacent precisely when the corresponding lines meet in a simplex point. This viewpoint is closely related to the Grassmann graphs of classes of linear codes studied by Cardinali, Giuzzi and Kwiatkowski \cite{CGK2021}; for the point-line geometry attached specifically to simplex codes, see \cite{KPT2026}.

The small cases are known. The graph $\Gs(2,2)$ has one vertex, while $\Gs(2,3)$ is isomorphic to $K_{4,4}$. For $q=4$, Kwiatkowski and Pankov determined the complete distance relation; it follows from their description that the diameter is $3$ \cite{KP2020}. Later work on simplex-code graphs focused mainly on maximal cliques and on the associated point-line geometry \cite{KP2024,KPT2026}. In a broader coding-theoretic setting, Cardinali and Giuzzi introduced and studied Grassmannian point-line geometries of classes of linear codes and their collinearity graphs \cite{CG2024}. None of these results determines the diameter of $\Gs(2,q)$ for a uniform range $q\ge5$.

The current state of the subject is summarized in Pankov's preliminary manuscript \emph{Notes on the graphs of non-degenerate linear codes} (2026, Section~10.3). The cases $\Gs(2,2)$, $\Gs(2,3)$ and $\Gs(2,4)$ are singled out there, while the distance problem for simplex-code graphs is left open in general (Problem~10.11). Here we address the more modest question of the diameter and determine it for every two-dimensional simplex-code graph.

\begin{theorem}\label{thm:main}
For every prime power $q$,
\[
\diam \Gs(2,q)=
\begin{cases}
0,&q=2,\\[1mm]
2,&q=3,\\[1mm]
3,&q\ge4.
\end{cases}
\]
\end{theorem}

The new part of Theorem~\ref{thm:main} is the assertion for $q\ge5$. We include $q=2,3,4$ both for completeness and to keep the paper self-contained; the argument for $q=4$ is a short alternative proof of the diameter consequence of the distance description in \cite{KP2020}. To the best of our knowledge, the diameter had not previously been determined for any uniform range $q\ge5$.

The proof becomes short once the appropriate intermediate statement is isolated. A consequence of Marshall Hall's theorem produces, from any two simplex points with different zero coordinates, a third simplex point collinear with both. This gives a path of length at most $3$ between arbitrary simplex lines. For $q\ge5$ we construct an explicit diagonal pair of simplex lines and rule out a common neighbor in two independent ways, first by moments and then by a product argument. For $q\ge7$ we also give an independent counting proof based on the number of vertices and the valency. The case $q=4$ is treated separately.

As an immediate consequence, for $q\ge4$ the graph $\Gs(2,q)$ is not isometrically embedded in its ambient Grassmann graph, whose diameter is $2$. Thus the two-dimensional case agrees with the non-isometric behaviour expected in the general discussion of simplex-code graphs.

\section{Simplex points, simplex lines, and the upper bound}
Throughout the paper, $V=\F_q^{q+1}$ and the coordinate positions are indexed by $I=\{0,1,\dots,q\}$.
For a simplex point $P=\langle p\rangle=\langle p_0,\dots,p_q\rangle$, let $z(P)$ denote the unique position at which $p$ is zero.

We will use the following elementary criterion throughout.

\begin{lemma}[Ratio criterion]\label{lem:ratio}
Let $P=\langle p\rangle$ and $Q=\langle q\rangle$ be distinct simplex points. Then $P$ and $Q$ lie on a simplex line if and only if $z(P)\ne z(Q)$ and, with $a=z(P)$ and $b=z(Q)$,
\[
\left\{\frac{q_i}{p_i}: i\in I\setminus\{a,b\}\right\}=\F_q^\times.
\]
\end{lemma}

\begin{proof}
The points of the projective line $\langle P,Q\rangle$ have representatives $\lambda p+\mu q$, with $[\lambda:\mu]\in\PP^1(\F_q)$. Its $i$th coordinate vanishes at the parameter $[-q_i:p_i]$. Therefore $\langle P,Q\rangle$ is a simplex line exactly when the $q+1$ coordinate functions vanish at $q+1$ distinct points of $\PP^1(\F_q)$. The two zero positions of $P$ and $Q$ must therefore be different, and after deleting them the remaining $q-1$ ratios are precisely the elements of $\F_q^\times$. The converse follows by reversing the argument.
\end{proof}

In particular, the $q+1$ points of a simplex line have pairwise distinct zero positions; hence each coordinate position occurs exactly once.

\begin{lemma}\label{lem:common-neighbor}
Let $L$ and $M$ be disjoint simplex lines. They have a common neighbor in $\Gs(2,q)$ if and only if there are points $P\in L$ and $Q\in M$ lying on a simplex line.
\end{lemma}

\begin{proof}
If a simplex line $N$ is adjacent to both $L$ and $M$, then $P=N\cap L$ and $Q=N\cap M$ are distinct and $N=\langle P,Q\rangle$. The converse is immediate.
\end{proof}

\paragraph{The bridge argument}
We need only the following consequence of a theorem of Marshall Hall \cite{Hall1952}. In multiplicative notation, Hall's theorem says that if $A$ is a finite abelian group of order $n$ and $b_1\cdots b_n=1$, then the elements of $A$ can be ordered as $x_1,\dots,x_n$ in such a way that $x_1b_1,\dots,x_nb_n$ are again all the elements of $A$.

\begin{lemma}[Partial Hall theorem]\label{lem:hall}
Let $A$ be a finite abelian group of order $n$, and let $b_1,\dots,b_{n-1}\in A$. There are pairwise distinct $x_1,\dots,x_{n-1}\in A$ for which $x_1b_1,\dots,x_{n-1}b_{n-1}$ are pairwise distinct.
\end{lemma}

\begin{proof}
Put $b_n=(b_1\cdots b_{n-1})^{-1}$ and apply Hall's theorem to $b_1,\dots,b_n$. Removing the last pair gives the required conclusion.
\end{proof}

\begin{lemma}[Bridge lemma]\label{lem:bridge}
Let $P,Q$ be simplex points with $z(P)\ne z(Q)$. Then there is a simplex point $R$ such that both $\langle P,R\rangle$ and $\langle R,Q\rangle$ are simplex lines.
\end{lemma}

\begin{proof}
Write $a=z(P)$ and $b=z(Q)$, choose $k\in I\setminus\{a,b\}$, and set $J=I\setminus\{a,b,k\}$, so $|J|=q-2$.
For $i\in J$, put $c_i=q_i/p_i\in\F_q^\times$. Apply Lemma~\ref{lem:hall} to the group $\F_q^\times$, which has order $q-1$, and to the $q-2$ elements $c_i^{-1}$. We obtain pairwise distinct $x_i\in\F_q^\times$ such that $y_i:=x_i c_i^{-1}$, $i\in J$, are pairwise distinct as well.

Let $x_b$ be the unique element of $\F_q^\times$ not among the $x_i$, and let $y_a$ be the unique element not among the $y_i$. Define $R=\langle r\rangle$ by $r_k=0$, $r_i=p_i x_i$ for $i\in J$, $r_b=p_bx_b$, and $r_a=q_ay_a$.
The resulting point $R$ is a simplex point. Moreover,
\[
\left\{\frac{r_i}{p_i}:i\ne a,k\right\}
=\{x_i:i\in J\}\cup\{x_b\}=\F_q^\times,
\]
so $P$ and $R$ lie on a simplex line. Likewise,
\[
\left\{\frac{r_i}{q_i}:i\ne b,k\right\}
=\{y_i:i\in J\}\cup\{y_a\}=\F_q^\times,
\]
and $R,Q$ lie on a simplex line as well.
\end{proof}

\begin{proposition}\label{prop:upper}
For every prime power $q$, $\diam\Gs(2,q)\le3$.
\end{proposition}

\begin{proof}
Let $L,M$ be simplex lines. Each simplex line contains exactly one point with its zero in each coordinate position. Choose $P\in L$ and $Q\in M$ with $z(P)\ne z(Q)$. By Lemma~\ref{lem:bridge} there is a simplex point $R$ such that $\langle P,R\rangle$ and $\langle R,Q\rangle$ are simplex lines. Therefore $L\;\mathrel{\text{---}}\;\langle P,R\rangle\;\mathrel{\text{---}}\;\langle R,Q\rangle\;\mathrel{\text{---}}\;M$ is a walk of length at most $3$. Deleting repetitions, if necessary, gives a path from $L$ to $M$ of length at most $3$.
\end{proof}

\section{The lower bound for \texorpdfstring{$q\ge5$}{q >= 5}}\label{sec:lower-bound}
\paragraph{Moment criterion}
For simplex points $P=\langle p\rangle$ and $Q=\langle q\rangle$ and $1\le m\le q-2$, put $\mu_m(P,Q)=\sum_{i=0}^{q}p_i^{q-1-m}q_i^m$.
The value of $\mu_m(P,Q)$ depends on the chosen representatives, but its vanishing does not: rescaling $p$ and $q$ multiplies $\mu_m(P,Q)$ by a non-zero scalar. We will need only the following necessary condition.

\begin{lemma}[Moment lemma]\label{lem:moments}
If $P$ and $Q$ lie on a simplex line, then $\mu_m(P,Q)=0$ for $1\le m\le q-2$.
\end{lemma}

\begin{proof}
Let $a=z(P)$ and $b=z(Q)$. The terms with indices $a,b$ vanish. For every remaining index,
\[
p_i^{q-1-m}q_i^m
=p_i^{q-1}\left(\frac{q_i}{p_i}\right)^m
=\left(\frac{q_i}{p_i}\right)^m.
\]
By Lemma~\ref{lem:ratio}, the ratios $q_i/p_i$ run through $\F_q^\times$. It follows that $\mu_m(P,Q)=0$, because $\sum_{x\in\F_q^\times}x^m=0$ for $1\le m\le q-2$.
\end{proof}

Index the coordinates by $\F_q\cup\{\infty\}$ and consider the standard simplex line $L$ with points
\[
A(t)=\bigl((t-a)_{a\in\F_q},1\bigr),\qquad t\in\F_q,
\]
and $A(\infty)=(1,\dots,1,0)$. Thus the zero position of $A(t)$ is $t$. Let $\xi$ be a generator of $\F_q^\times$, choose a three-element set $S$ of coordinate positions, and put
\[
D=\operatorname{diag}(d_k)_{k\in\F_q\cup\{\infty\}},\qquad
 d_k=\begin{cases}\xi,&k\in S,\\1,&k\notin S.\end{cases}
\]
Set $M=D(L)$. The diagonal map preserves zero positions, so $M$ is again a simplex line. Moreover, $L$ and $M$ are disjoint. Indeed, if $A(i)$ and $DA(j)$ represented the same projective point, their zero positions would force $i=j$. But $DA(i)$ cannot be a scalar multiple of $A(i)$, since among the coordinates different from $i$ there remain diagonal entries equal to both $1$ and $\xi$.

\begin{proposition}\label{prop:qge5}
If $q\ge5$, then $L$ and $M$ are at distance $3$. Consequently, $\diam\Gs(2,q)=3$.
\end{proposition}

\begin{proof}[Moment proof]
Suppose first that $i\ne j$, and put
\[
u_k=\frac{A_k(j)}{A_k(i)},\qquad k\ne i,j.
\]
Since $A(i)$ and $A(j)$ lie on the simplex line $L$, Lemma~\ref{lem:ratio} gives
\[
\{u_k:k\ne i,j\}=\F_q^\times.
\]
For $1\le m\le q-2$, the corresponding moment for $A(i)$ and $DA(j)$ is therefore
\[
\mu_m\bigl(A(i),DA(j)\bigr)
=(\xi^m-1)\sum_{k\in S\setminus\{i,j\}}u_k^m. \tag{3.1}
\]
This follows because the sum of the $m$th powers of all elements of $\F_q^\times$ is zero, while $D$ changes only the coordinates in $S$.

Let $T=S\setminus\{i,j\}$ and $r=|T|$. Then $1\le r\le3$. Since $q\ge5$, the moments with $1\le m\le r$ are available, and since $\xi$ has order $q-1\ge4$, one has $\xi^m\ne1$ for these values of $m$. If all these moments vanished, then
\[
\sum_{k\in T}u_k^m=0,\qquad m=1,\dots,r.
\]
The elements $u_k$, $k\in T$, are distinct and non-zero. After ordering $T=\{k_1,\dots,k_r\}$, the determinant of the coefficient matrix $(u_{k_a}^m)_{1\le m,a\le r}$ is, up to sign,
\[
\left(\prod_{a=1}^r u_{k_a}\right)
\prod_{1\le a<b\le r}(u_{k_b}-u_{k_a}),
\]
which is non-zero because the $u_{k_a}$ are non-zero and pairwise distinct. The displayed equations would therefore force this matrix to send the non-zero vector $(1,\dots,1)^T$ to zero, which is impossible. Hence at least one of the moments in (3.1) is non-zero.

By Lemma~\ref{lem:moments}, $A(i)$ and $DA(j)$ cannot lie on a simplex line. The case $i=j$ is excluded by Lemma~\ref{lem:ratio}, because the two points have the same zero position. Therefore no point of $L$ is collinear with a point of $M$. By Lemma~\ref{lem:common-neighbor}, the lines $L$ and $M$ have no common neighbor. Thus $\dist(L,M)\ge3$, and Proposition~\ref{prop:upper} gives equality.
\end{proof}

\begin{proof}[Product proof]
For distinct $i,j\in\F_q\cup\{\infty\}$,
\[
\prod_{k\ne i,j}\frac{A_k(j)}{A_k(i)}=-1. \tag{3.2}
\]
For finite $i,j$ this follows from the derivative identity for $X^q-X$,
\[
\prod_{a\in\F_q\setminus\{x\}}(x-a)=-1.
\]
If $j=\infty$, the product in (3.2) is the reciprocal of this expression with $x=i$, whereas for $i=\infty$ it is the same expression with $x=j$. Hence (3.2) holds in all cases.

If $A(i)$ and $DA(j)$ lay on a simplex line, with $i\ne j$, the $q-1$ ratios on their common non-zero coordinates would be the elements of $\F_q^\times$, whose product is $-1$ (with $-1=1$ in characteristic $2$). Since $\prod_k d_k=\xi^3$, identity (3.2) gives instead
\[
\prod_{k\ne i,j}\frac{d_kA_k(j)}{A_k(i)}
=-\frac{\xi^3}{d_id_j}.
\]
Comparing the two products would force $d_id_j=\xi^3$. But $d_i,d_j\in\{1,\xi\}$, so $d_id_j\in\{1,\xi,\xi^2\}$, whereas $\xi^3\notin\{1,\xi,\xi^2\}$ because $\xi$ has order at least $4$. This contradiction shows that no such pair is collinear. The case $i=j$ is again excluded by Lemma~\ref{lem:ratio}. Lemma~\ref{lem:common-neighbor} and Proposition~\ref{prop:upper} now give $\dist(L,M)=3$.
\end{proof}

\begin{remark}\label{rem:q4-failure}
The restriction $q\ge5$ in the two preceding arguments is essential. For $q=4$, a generator $\xi\in\F_4^\times$ has order $3$, so $\xi^3=1$. If $i,j$ are the two coordinate positions outside $S$, then $d_id_j=1=\xi^3$, and the product argument gives no contradiction. The moment argument fails for a different reason: here $q-2=2$, so only $\mu_1$ and $\mu_2$ are available, while $T=S$ has three elements. For the same choice of $i,j$, the ratios $u_k$, $k\in S$, are precisely the three elements of $\F_4^\times$, and hence
\[
\sum_{k\in S}u_k=\sum_{k\in S}u_k^2=0.
\]
Thus both available moments vanish, and $q=4$ requires a separate construction, which we give in Section~\ref{sec:small-fields}.
\end{remark}

\paragraph{Alternative counting proof for \texorpdfstring{$q\ge7$}{q >= 7}}
For the alternative counting proof we need only the following elementary counts. Set $r=(q-1)!$.

\begin{lemma}\label{lem:counts}
The number of simplex points is $N_q=(q+1)(q-1)^{q-1}$. Every simplex point lies on exactly $r$ simplex lines. Consequently, $|V(\Gs(2,q))|=(q-1)^{q-1}r$, and $\Gs(2,q)$ is regular of valency $d_q=(q+1)(r-1)$.
\end{lemma}

\begin{proof}
For a fixed zero position there are $(q-1)^q$ such non-zero coordinate vectors, and passing to projective points divides this number by $q-1$. Hence $N_q=(q+1)(q-1)^{q-1}$.

For the number of simplex lines through a point, permuting coordinates and rescaling the non-zero coordinates independently sends any simplex point to $P=\langle1,\dots,1,0\rangle$ and preserves simplex points and simplex lines. Every simplex line through $P$ has exactly $q$ further points; their last coordinate is non-zero, so each has a unique representative of the form $Q=(a_0,\dots,a_{q-1},1)$. The $q$ points $\langle Q+\lambda P\rangle$, $\lambda\in\F_q$, are simplex points precisely when $a_0,\dots,a_{q-1}$ are pairwise distinct, hence form a permutation of $\F_q$. There are therefore $q!$ possible points $Q$, and each line through $P$ contributes exactly $q$ of them. Therefore the number of such lines is $q!/q=r$.

The remaining formulas follow from straightforward incidence counts. Since every simplex line contains $q+1$ simplex points, $|V(\Gs(2,q))|(q+1)=N_qr$, and hence $|V(\Gs(2,q))|=(q-1)^{q-1}r$. Finally, through each of the $q+1$ points of a fixed simplex line there are $r-1$ other simplex lines, and no neighboring line is counted twice. Hence $d_q=(q+1)(r-1)$.
\end{proof}

For completeness, we include the independent counting proof for $q\ge7$.

\begin{proposition}\label{prop:largeq}
If $q\ge7$, then $\diam\Gs(2,q)=3$.
\end{proposition}

\begin{proof}
Suppose, for a contradiction, that $\diam\Gs(2,q)\le2$. From a fixed vertex there is one vertex at distance $0$, at most $d_q$ vertices at distance $1$, and at most $d_q(d_q-1)$ vertices at distance $2$. Thus
\[
|V(\Gs(2,q))|\le 1+d_q+d_q(d_q-1)=1+d_q^2. \tag{3.3}
\]
We now show that the reverse strict inequality holds for $q\ge7$.

Define $f(q)=(q-1)^{q-1}/((q+1)^2r)$. At $q=7$ we have $f(7)=81/80>1$. When $q$ is replaced by $q+1$, the corresponding value of $r$ is $qr$, and hence
\[
\frac{f(q+1)}{f(q)}
=\left(1+\frac1{q-1}\right)^{q-1}
 \left(\frac{q+1}{q+2}\right)^2
>2\left(\frac89\right)^2>1
\]
for every $q\ge7$. It follows that $(q-1)^{q-1}>(q+1)^2r$. Using Lemma~\ref{lem:counts},
\[
\begin{aligned}
|V(\Gs(2,q))|
  &=(q-1)^{q-1}r\\
  &>(q+1)^2r^2\\
  &>1+(q+1)^2(r-1)^2
   =1+d_q^2,
\end{aligned}
\]
contradicting (3.3). Hence the diameter is at least $3$, and Proposition~\ref{prop:upper} gives equality.
\end{proof}

\section{The small fields}\label{sec:small-fields}
We finish with the three small fields $q=2,3,4$.

\paragraph{The cases \texorpdfstring{$q=2$ and $q=3$}{q = 2 and q = 3}}
For $q=2$ there is a unique binary simplex code of dimension $2$. Hence $\Gs(2,2)=K_1$, and its diameter is $0$.

\begin{proposition}\label{prop:q3}
The graph $\Gs(2,3)$ is isomorphic to $K_{4,4}$. In particular, $\diam\Gs(2,3)=2$.
\end{proposition}

\begin{proof}
Lemma~\ref{lem:counts} gives $8$ vertices and valency $4$; each simplex point lies on exactly two simplex lines. It remains only to rule out triangles.

Three pairwise adjacent simplex lines cannot pass through one point, since only two simplex lines pass through any simplex point. Suppose instead that three distinct simplex lines form a triangle with distinct intersection points $P,Q,R$. Their zero positions are pairwise different. After a permutation of coordinates and a rescaling of representatives, write
\[
P=(0,a,b,1),\qquad Q=(c,0,d,1),\qquad R=(e,f,0,1),
\]
with $a,b,c,d,e,f\in\F_3^\times$. Since $P,Q$ lie on a simplex line, Lemma~\ref{lem:ratio} gives $d=-b$. The other two pairs similarly give $e=-c$ and $f=-a$. Thus $P+Q+R=0$, so $P,Q,R$ are projectively collinear, contrary to the assumption that the three intersection points are distinct. Therefore $\Gs(2,3)$ is triangle-free.

Fix a vertex $L$. Its four neighbors form an independent set. Only three vertices remain. Each neighbor of $L$ has degree $4$, is adjacent to $L$, and cannot be adjacent to another neighbor of $L$; it must therefore be adjacent to all three remaining vertices. Hence $\Gs(2,3)\cong K_{4,4}$.
\end{proof}

\paragraph{The case \texorpdfstring{$q=4$}{q = 4}}
Kwiatkowski and Pankov determined the full distance relation for $q=4$ \cite{KP2020}; their description implies that the diameter is $3$. For completeness, we give a short moment proof of this consequence.

Write $\F_4=\{0,1,\alpha,\alpha+1\}$, where $\alpha^2=\alpha+1$, and index the coordinates by $\F_4\cup\{\infty\}$ in the order $0,1,\alpha,\alpha+1,\infty$. Consider the lines $L,M$ parametrized by
\[
\begin{aligned}
A(t)&=(t,t+1,t+\alpha,t+\alpha+1,1), &&t\in\F_4,\\
A(\infty)&=(1,1,1,1,0),
\end{aligned}
\]
and
\[
\begin{aligned}
B(u)&=(u,u+1,\alpha u+1,\alpha u+\alpha+1,\alpha+1), &&u\in\F_4,\\
B(\infty)&=(1,1,\alpha,\alpha,0).
\end{aligned}
\]
Since $A(t)=A(0)+tA(\infty)$ and $B(u)=B(0)+uB(\infty)$, the two parametrizations describe projective lines. The zero position of $A(t)$ is $t$, while that of $B(u)$ is $u^2$; the points at infinity have zero position $\infty$. Hence every point on either line is a simplex point, and $B(s^2)$ and $A(s)$ have the same zero position.

These two lines are disjoint. Every finite point of either line has non-zero last coordinate, whereas both points at infinity have last coordinate $0$, so a finite point cannot be projectively equal to a point at infinity. Among finite points, the sum of the coordinates of every $B(u)$ is $0$, whereas the corresponding sum for $A(t)$ is $1$; finally, $A(\infty)$ is not projectively equal to $B(\infty)$.

For $s,t\in\F_4$, direct calculation gives
\[
\mu_1\bigl(B(s^2),A(t)\bigr)=\alpha(s+t). \tag{4.1}
\]
This is non-zero when $s\ne t$. The pairs involving the infinite parameter satisfy
\[
\mu_1\bigl(B(s^2),A(\infty)\bigr)=\alpha,
\qquad
\mu_1\bigl(B(\infty),A(t)\bigr)=\alpha. \tag{4.2}
\]
By Lemma~\ref{lem:moments}, no point of $L$ is collinear with a point of $M$ having a different zero position; pairs with the same zero position are excluded by Lemma~\ref{lem:ratio}. Lemma~\ref{lem:common-neighbor} therefore shows that $L$ and $M$ have no common neighbor. Hence $\dist(L,M)\ge3$, and Proposition~\ref{prop:upper} yields $\diam\Gs(2,4)=3$.

\begin{proof}[Proof of Theorem~\ref{thm:main}]
Proposition~\ref{prop:upper} gives the upper bound $3$ for every $q$, while Proposition~\ref{prop:qge5} gives equality for every $q\ge5$. Section~\ref{sec:small-fields} gives the remaining cases $q=2,3,4$. Together, these cases cover all prime powers and complete the proof of Theorem~\ref{thm:main}.
\end{proof}

\begin{remark}
For $q\ge4$, Theorem~\ref{thm:main} implies that $\Gs(2,q)$ is not isometrically embedded in the ambient Grassmann graph $\G_2(\F_q^{q+1})$, whose diameter is $2$. The graph $\Gs(2,3)\cong K_{4,4}$ is the known exceptional non-trivial two-dimensional case.
\end{remark}

\section*{Statements and Declarations}
\noindent\textbf{Funding.} No funding was received for this research.

\noindent\textbf{Competing interests.} The author declares that they have no known competing financial interests or personal relationships that could have appeared to influence the work reported in this paper.

\noindent\textbf{Author contributions.} Artur Siemaszko is the sole author of this work and prepared the manuscript.

\noindent\textbf{Data availability.} No data was used for the research described in this article.

\noindent\textbf{Use of generative AI.} During the preparation of this manuscript, the author used OpenAI's ChatGPT to test mathematical hypotheses, critically review mathematical arguments, and assist with the organization, language editing, and LaTeX preparation of the text. The author reviewed and verified the mathematical results, proofs, and references and takes full responsibility for the content of the manuscript.


\begin{thebibliography}{99}

\bibitem{CG2024}
I.~Cardinali, L.~Giuzzi,
Grassmannians of codes,
\emph{Finite Fields Appl.} 94 (2024), 102342.
\newblock \url{https://doi.org/10.1016/j.ffa.2023.102342}.

\bibitem{CGK2021}
I.~Cardinali, L.~Giuzzi, M.~Kwiatkowski,
On the Grassmann graph of linear codes,
\emph{Finite Fields Appl.} 75 (2021), 101895.
\newblock \url{https://doi.org/10.1016/j.ffa.2021.101895}.

\bibitem{Hall1952}
M.~Hall, Jr.,
A combinatorial problem on abelian groups,
\emph{Proc. Amer. Math. Soc.} 3 (1952), 584--587.
\newblock \url{https://doi.org/10.1090/S0002-9939-1952-0050579-7}.

\bibitem{KP2016}
M.~Kwiatkowski, M.~Pankov,
On the distance between linear codes,
\emph{Finite Fields Appl.} 39 (2016), 251--263.
\newblock \url{https://doi.org/10.1016/j.ffa.2016.02.004}.

\bibitem{KP2020}
M.~Kwiatkowski, M.~Pankov,
The graph of 4-ary simplex codes of dimension 2,
\emph{Finite Fields Appl.} 67 (2020), 101709.
\newblock \url{https://doi.org/10.1016/j.ffa.2020.101709}.

\bibitem{KP2024}
M.~Kwiatkowski, M.~Pankov,
On maximal cliques in the graph of simplex codes,
\emph{J. Geom.} 115 (2024), Article 10.
\newblock \url{https://doi.org/10.1007/s00022-023-00709-y}.

\bibitem{KPP2018}
M.~Kwiatkowski, M.~Pankov, A.~Pasini,
The graphs of projective codes,
\emph{Finite Fields Appl.} 54 (2018), 15--29.
\newblock \url{https://doi.org/10.1016/j.ffa.2018.07.003}.

\bibitem{KPT2026}
M.~Kwiatkowski, M.~Pankov, A.~Tyc,
One class of maximal cliques in the collinearity graphs of geometries related to simplex codes,
\emph{Finite Fields Appl.} 111 (2026), 102784.
\newblock \url{https://doi.org/10.1016/j.ffa.2025.102784}.

\bibitem{Pankov2023}
M.~Pankov,
The graphs of non-degenerate linear codes,
\emph{J. Combin. Theory Ser. A} 195 (2023), 105720.
\newblock \url{https://doi.org/10.1016/j.jcta.2022.105720}.

\end{thebibliography}
\end{document}